\documentclass[11pt,psamsfonts,fleqn,leqno]{amsart}
\usepackage{amsfonts}
\usepackage{mathrsfs,hyperref}
\usepackage{a4wide,amssymb,amsmath,amsthm,xspace,epsfig}
\usepackage{graphicx, color}
\usepackage{rotating}
\usepackage{graphicx, color}
\usepackage[mathscr]{euscript}
\def\B{\mathbb B}

\def\R{\mathbb R}
\def\S{\mathbb S}

\def\*{\times}
\def\nix{\varnothing}

\def\la{\langle}
\def\ra{\rangle}
\def\ol{\overline}
\def\sm{\setminus}

\def\del{\partial}

\newtheorem{thm}{Theorem}
\newtheorem{defn}[thm]{Definition}
\newtheorem{ques}[thm]{Question}

\newtheorem{cor}[thm]{Corollary}

\newtheorem{prop}[thm]{Proposition}

\newtheorem{rem}[thm]{Remark}

\begin{document}

\title{Handlebody Neighbourhoods and a conjecture of Adjamagbo}
\author{David Gauld}
\address{The University of Auckland}
  \email{d.gauld@auckland.ac.nz}

\date{\today}

\begin{abstract}
Using handlebodies we verify a conjecture of P Adjamagbo that if the frontier of a relatively compact subset $V_0$ of a manifold is a submanifold then there is an increasing family $\{V_r\}$ of relatively compact open sets indexed by the positive reals so that the frontier of each is a submanifold, their union is the whole manifold and for each $r\ge 0$ the subfamily indexed by $(r,\infty)$ is a neighbourhood basis of the closure of the $r^{\rm th}$ set.
\end{abstract}
\maketitle

 \subjclass{}

\keywords{}

  
 
 \section{Introduction}
 
Throughout this paper we use the euclidean metric on $\R^n$. All of our manifolds are assumed to be metrisable. No extra structure on the manifold is assumed.

Pascal Adjamagbo, \cite{A21}, has proposed the following conjecture\footnote{I thank Adjamagbo for drawing my attention to this conjecture.}:
\begin{quote}
Given a relatively compact non-empty open subset $V_0$ of a connected topological manifold $M^m$ such that the boundary of $V_0$ is a manifold, there exists an increasing family $\la V_r\ra_{r\in[0,\infty)}$ of relatively compact open subsets of $M$ the boundaries of which are topological manifolds such that $M$ is the union of the elements of the family, and that for any $r \in[0,\infty)$, the family $\la V_{r'}\ra_{r'>r}$ is a fundamental system of neighbourhoods of the closure of $V_r$.
\end{quote}
Adjamagbo makes no assumptions regarding the tameness of the boundary manifold, so it could be wild at every point.

In this paper we verify Adjamagbo's conjecture.

An important tool used in our proof is a handlebody, a special structure on (part of) a manifold.

\begin{defn}
Let $M^m$ be a topological manifold with boundary and let $k\in\{0,1,\dots,m\}$. Suppose that $e:\S^{k-1}\*\B^{m-k}\to\del M$ is an embedding and let $M_e$ be the $m$-manifold obtained from the disjoint union of $M$ and $\B^k\*\B^{m-k}$ by identifying $x\in\S^{k-1}\*\B^{m-k}$ with $e(x)\in\del M$. Then we say that $M_e$ is obtained from $M$ by \emph{adding a $k$-handle of dimension $m$ to $M$}, with the prefix $k$ suppressed when we do not want to specify it. The image of $\B^k\*\B^{m-k}$ in $M_e$ is called a \emph{($k$)-handle}. A \emph{handlebody} is a manifold obtained inductively beginning at $\nix$ then successively adding a handle of dimension $m$ to the output of the previous step. If a handlebody has been obtained by adding only finitely many handles then we call it a \emph{finite handlebody}.
\end{defn}
In this definition we take $\B^\ell=\{x\in\R^\ell\ /\ |x|\le1\}$, the unit ball in $\R^\ell$, so $\S^{\ell-1}$ is the boundary sphere of $\B^\ell$. Of course when constructing a handlebody, of necessity the first handle to be added must be a 0-handle as $\del\nix=\nix$ and $\S^{k-1}\not=\nix$ when $k>0$. A handlebody is a manifold with boundary.

The following two theorems characterise the existence of handlebody structures on topological manifolds.

\begin{thm}{\rm\cite[Theorem 9.2]{FQ90}}
A metrisable manifold fails to have a handlebody decomposition if and only if it is an unsmoothable 4-manifold.
\end{thm}

\begin{thm}{\rm\cite[Theorem 8.2]{FQ90}}
Every connected, non-compact, metrisable 4-manifold is smoothable.
\end{thm}

A basic result needed in our proof is the following proposition.

\begin{prop}{\rm\cite[Proposition 3.17]{G14}}\label{good handlebody}
Suppose that $M$ is a handlebody and $K\subset M$ is compact. Then there is a finite handlebody $W\subset M$ which is a neighbourhood of $K$.
\end{prop}

We also require the following weak version of the collaring theorem of Brown.

\begin{thm}{\rm\cite{B62}}\label{collared boundary}
Let $W$ be a finite handlebody. Then there is an embedding $e:\del W\*[a,b]\to W$ ($a<b$) such that $e(x,b)=x$ for each $x\in\del W$.
\end{thm}
The embedding $e$ is called a \emph{collar} of the boundary $\del W$. Using the notation of Theorem \ref{collared boundary}, we will call the set $e(c)$ a \emph{level} of the collar and the subset $W\sm e(\del W\*(c,b])$ will be said to be \emph{inside the level $e(c)$}. A set inside a level of a handlebody is a compact subset of $W$; moreover the boundary of this set is a manifold of one lower dimension and, when $c>a$, the set inside the level $e(c)$ is homeomorphic to $W$ so is itself a finite handlebody.
 
 \section{Handlebody Neighbourhoods and Adjamagbo's Conjecture}
 
In this section we construct a neighbourhood basis of a compactum in a manifold where the neighbourhoods making up the basis are all handlebodies. We then use this construction to prove Adjamagbo's conjecture.

\begin{thm}\label{Handlebody neighbourhoods on manifolds}
Let $M^m$ be a connected manifold and $V_0\subset M$ a non-empty, relatively compact, open subset of $M$ such that the frontier of $V_0$ in $M$ is an $(m-1)$-manifold. Suppose further in the case where $M$ is closed that $\ol{V_0}\not= M$. Then for each real number $r\in(0,1)$ there is a finite handlebody $W_r$ such that for each $r\in[0,1)$, the collection $\{{\rm Int}(W_{r'})\ /\ r'>r\}$ is a neighbourhood basis of $W_r$, where $W_0$ is the closure of $V_0$.
\end{thm}
\begin{proof}
Suppose given $M^m$ and $V_0\subset M$ as in Theorem \ref{Handlebody neighbourhoods on manifolds}. In the case where $M$ is closed pick a point $p\in M\sm\ol{V_0}$. It follows that $M$ (in the case where $M$ is open) or $M\sm\{p\}$ (in the case where $M$ is closed) may be embedded properly in some euclidean space, see \cite[Theorem 2.1(1$\Leftrightarrow$36)]{G14} for example. Let $d$ be the metric on $M$ or $M\sm\{p\}$ as appropriate inherited from the euclidean metric under some fixed proper embedding.

For each natural number $n$ let 
$$U_n=\left\{x\in M\ /\ d(x,\ol{V_0})<\frac1n\right\}.$$
Then $\ol{U_{n+1}}\subset U_n$ for each $n$ and the collection $\{U_n\ /\ n=1,2,\dots\}$ is a neighbourhood basis for $\ol{V_0}$.

For each $n$ use Proposition \ref{good handlebody} to find a finite handlebody $X_n\subset U_n$ containing $\ol{U_{n+1}}$ in its interior. By Theorem \ref{collared boundary} we may find a collar $e_n:\del X_n\*\left[\frac1{n+2},\frac1n\right]\to X_n$ such that $e_n\left(x,\frac1n\right)=x$ for each $x\in\del X_n$. Further, compactness of the disjoint sets $\ol{U_{n+1}}$ and $\del X_n$ allows us to assume that the image of $e_n$ is disjoint from $\ol{U_{n+1}}$.

For each $r\in(0,1)$ choose $W_r$ as follows. There is a unique natural number $n$ such that $\frac1{n+1}\le r<\frac1n$. Let $W_r$ be the set inside the level $e_n(r)$ of the handlebody $X_n$. As noted above, $W_r$ is a handlebody. Moreover, for each $r\in[0,1)$, the collection ${\rm Int}(W_{r'})\ /\ r'>r\}$ is a neighbourhood basis of $W_r$.
\end{proof}

We are now ready to prove Adjamagbo's conjecture.

\begin{cor}\label{Adjamagbo's conjecture}
Let $M^m$ be a connected topological manifold and $V_0\subset M$ be a relatively compact non-empty open subset $V_0$ of $M$ such that the boundary of $V_0$ is a manifold. Then there exists a family $\{V_r\ /\ r\in[0,\infty)\}$ of relatively compact open subsets of $M$ the boundaries of which are topological manifolds such that $M$ is the union of the elements of the family, and that for any $r \in[0,\infty)$ the family $\{V_{r'}\ /\ r'>r\}$ is a neighbourhood basis of the closure of $V_r$.
\end{cor}
\begin{proof}
The case where $\ol{V_0}=M$ is trivial so we assume that $\ol{V_0}\not= M$.

Applying Theorem \ref{Handlebody neighbourhoods on manifolds} to the case $\ol{V_0}\not= M$, for each $r\in(0,1)$ set $V_r={\rm Int}(W_r)$. Then each $V_r$ is open and relatively compact with boundary a topological manifold, and for each $r<1$ the family $\{V_{r'}\ /\ r'>r\}$ is a neighbourhood basis of the closure of $V_r$. The proof is complete in the case where $M$ is closed if we set $V_r=M$ for each $r\ge1$.

Suppose $M$ is open. In this case follow the procedure in the proof of Theorem \ref{Handlebody neighbourhoods on manifolds} but now replace the sets $U_n$ by sets $U_n'=\left\{x\in M\ /\ d(x,\ol{V_0})<n\right\}$ and the handlebodies $X_n$ by finite handlebodies $X_n'$ whose boundaries lie in $U_{n+1}'\sm\ol{U_n'}$. For each natural number $n$ and each $r\in[n,n+1)$ we then construct the open sets $V_r$ to be the interiors of sets inside appropriate levels of the handlebody $X_n'$.
\end{proof}

\begin{rem}
If $V_0$ is connected in Corollary \ref{Adjamagbo's conjecture} then we may also assume that each of the neighbourhoods $V_r$ is also connected. To ensure this, when we construct the handlebodies $X_n$, we discard any supernumerary components of the handlebodies.
\end{rem}

\begin{rem}
The assumption in Corollary \ref{Adjamagbo's conjecture} that $M$ is connected can be weakened to the assumption that $M$ has countably many components. We may construct sets such as the sets $V_r$ in the proof of Corollary \ref{Adjamagbo's conjecture} within each component of $M$, beginning with an arbitrary point of the component in place of $\ol V_0$ when a component contains no point of $V_0$, and then taking $V_r$ to be the union of the resulting sets.
\end{rem}

\section{Questions}

In our construction of the neighbourhoods $V_r$ in the proof of Corollary \ref{Adjamagbo's conjecture} the neighbourhoods grow gradually for a while then jump: for example $V_1\sm\cup_{r<1}V_r$ is non-empty. So the following question might be addressed.

\begin{ques}
Given a relatively compact non-empty open subset $V_0$ of a topological manifold $M$ such that $\del V_0$ is a manifold, is it possible to find a family $\{V_r\ /\ r\in[0,\infty)\}$ of relatively compact open subsets of $M$ the boundaries of which are topological manifolds such that $M=V_0\cup(\cup_{r\in[0,\infty)}\del V_r)$, and that for any $r \in[0,\infty)$ the family $\{V_{r'}\ /\ r'>r\}$ is a fundamental system of neighbourhoods of the closure of $V_r$? (Here $\del V_0$ denotes the boundary of $V_0$.)
\end{ques}

We might also explore Adjamagbo's conjecture in other categories. As an example:

\begin{ques}
Given a relatively compact non-empty open subset $V_0$ of a connected smooth manifold $M^m$ such that the boundary of $V_0$ is a smooth submanifold, there exists an increasing family $\la V_r\ra_{r\in[0,\infty)}$ of relatively compact open subsets of $M$ the boundaries of which are smooth submanifolds such that $M$ is the union of the elements of the family, and that for any $r \in[0,\infty)$, the family $\la V_{r'}\ra_{r'>r}$ is a fundamental system of neighbourhoods of the closure of $V_r$.
\end{ques}

Morse functions may well play a role in answering the latter question.

\end{document}